\documentclass[a4paper,12pt,oneside]{amsart} 

\usepackage[english]{babel}
\usepackage[latin9]{inputenc}
\usepackage{amsmath,amsthm,amssymb,amsfonts,amscd,amsbsy,amsxtra}
\usepackage{ucs}
\usepackage{graphicx}
\usepackage{paralist}
\usepackage{hyperref}

\newcommand{\R}{\mathbb{R}}

\newcommand{\N}{\mathbb{N}}
\newcommand{\C}{\mathbb{C}}

\newcommand{\ssubset}{\subset\joinrel\subset}

\usepackage{scalerel}[2014/03/10]
\usepackage[usestackEOL]{stackengine}
\def\avgint{\,\ThisStyle{\ensurestackMath{%
  \stackinset{c}{.2\LMpt}{c}{.5\LMpt}{\SavedStyle-}{\SavedStyle\phantom{\int}}}%
  \setbox0=\hbox{$\SavedStyle\int\,$}\kern-\wd0}\int}

\DeclareMathOperator{\Real}{Re}
\DeclareMathOperator{\dist}{dist}

\theoremstyle{plain}
\newtheorem{theorem}{Theorem}[section]
\newtheorem{proposition}[theorem]{Proposition}
\newtheorem{corollary}[theorem]{Corollary}
\newtheorem{lemma}[theorem]{Lemma}
\theoremstyle{definition}
\newtheorem{example}[theorem]{Example}
\newtheorem{remark}[theorem]{Remark}

\newtheorem{openproblem}[theorem]{Open Problem}

\title[Perron Solutions]{The Perron Solution for Vector-Valued Equations}
 
\author{\sc M. Kreuter}
\address{Marcel Kreuter\\Institute of Applied Analysis\\Ulm University\\89069 Ulm\\Germany}
\email{marcel.kreuter@uni-ulm.de}

\date{\today}


\begin{document}
\maketitle

\begin{abstract}
Given a continuous function on the boundary of a bounded open set in $\R^d$ there exists a unique bounded harmonic function, called the \emph{Perron solution}, taking the prescribed boundary values at least at all regular points (in the sense of Wiener) of the boundary. We extend this result to vector-valued functions and consider several methods of constructing the Perron solution which are classical in the real-valued case. We also apply our results to solve elliptic and parabolic boundary value problems of vector-valued functions.
\end{abstract}

\section{Problem Setting, Existence and Uniqueness}\label{X}
Let $\Omega\subset\R^d$ be an open and bounded set and let $X$ be a (real) Banach space. The set of all \emph{harmonic functions} will be denoted by $\mathcal{H}(\Omega,X):=\{u\in C^2(\Omega,X),\Delta u=0\}$. Given a function $f\in C(\partial\Omega,X)$ we consider the classical Dirichlet problem\\
\begin{align*}
\begin{cases}
u\in\mathcal{H}(\Omega,X)\cap C(\overline{\Omega},X)\\
u_{|\partial\Omega}=f.
\end{cases}
\end{align*}
\vspace{0.2cm}\\
A function $u$ satisfying the above is called a \emph{classical solution}. Note that the range of $f$ is separable. Hence we will without loss of generality assume that $X$ is separable, since we may restrict our arguments to the smallest subspace of $X$ containing the range of $f$.\\

In the case $X=\R$, it is well known that a generalized solution, the \emph{Perron solution}, exists. This solution is the unique function which is bounded, harmonic and satisfies the boundary values at least on the set $\partial_\textnormal{reg}\Omega$ of all \emph{regular points} of $\partial\Omega$, c.f. \cite[Chapter 6.6]{ArmitageGardinerPotentialTheory}, \cite[Chapter 2]{HelmsPotentialTheory} or \cite{KeldysDirichletproblem}. The set $\partial_\textnormal{reg}\Omega$ can be described by potential theoretic means, e.g. a point $z\in\partial\Omega$ is regular if and only if $\Omega^c$ is not thin at $z$ which in turn can be described via Wiener's criterion \cite[Theorems 7.5.1 and 7.7.2]{ArmitageGardinerPotentialTheory}. The set of all bounded harmonic functions will be denoted by $\mathcal{H}_b(\Omega,X)$.

\begin{theorem}\label{Perronsolution EnU}
Let $\Omega\subset\R^d$ be open and bounded and let $X$ be a real Banach space. For every $f\in C(\partial\Omega,X)$ there exists a unique function $H_f$, called the \emph{Perron solution}, which satisfies
\begin{align*}
\begin{cases}
H_f\in\mathcal{H}_b(\Omega,X)\\
\lim_{\xi\rightarrow z}H_f(\xi)=f(z)\textnormal{ for all }z\in\partial_\textnormal{reg}\Omega.
\end{cases}
\end{align*}
\end{theorem}

\begin{remark}
The real-valued Dirichlet problem has a classical solution for every boundary data if and only if $\partial_\textnormal{reg}\Omega=\partial\Omega$. It follows that if the Dirichlet problem has a solution for every real-valued continuous boundary data, then it also has a solution for every $X$-valued continuous boundary data.
\end{remark}

To prove this we will need some auxiliary results. The $(d-1)$-dimensional Hausdorff measure will be denoted by $\sigma_{d-1}$.

\begin{theorem}\cite[Lemma 5.1, Proposition 5.3  \& Theorem 5.4]{ArendtVector-valuedharmonic}\label{Arendt}
\begin{compactenum}[\upshape (i)]
\item\label{Arendt very weak} Let $u:\Omega\rightarrow X$ be locally bounded and suppose that there exists a separating subset $W\subset X'$ such that $\langle u,x'\rangle:=\langle u(\cdot),x'\rangle$ is harmonic for all $x'\in W$. Then $u$ is harmonic.
\item\label{Arendt Poisson} If $u\in\mathcal{H}(\Omega,X)$ then \emph{Poisson's Integral Formula}
\begin{align*}
u(x)={\frac{1}{r_0\sigma_{d-1}(B(0,1))}}\int_{\partial B(x_0,r_0)}{\frac{r_0^2-|x-x_0|^2}{|x-s|^d}u(s)\,d\sigma_{d-1}(s)}
\end{align*}
holds for all $x_0\in\Omega$ and $r_0>0$ such that $B(x_0,r_0)\ssubset\Omega$.
\item\label{Arendt nets} Let $(u_i:\Omega\rightarrow X)_{i\in I}$ be a bounded net of harmonic functions. Suppose that $(u_i(x))_{i\in I}$ converges for each $x\in\Omega$, then $(u_i)_{i\in I}$ converges uniformly on compact subsets of $\Omega$ and the limit is a harmonic function.
\end{compactenum}
\end{theorem}

If $K\subseteq X$ then $\overline{\textnormal{conv}}(K)$ will denote the \emph{closed convex hull} of $K$. Note that if $K$ is compact, then so is $\overline{\textnormal{conv}}(K)$, \cite[Theorem 5.35]{AlipranitsBorderInfinitedimensional}.

\begin{proposition}[Maximum Principle for vector-valued functions]\label{weakmax}
Let $u\in\mathcal{H}(\Omega,X)\cap C(\overline{\Omega},X)$. Then for all $\xi\in\Omega$ we have
\begin{compactenum}[\upshape (i)]
\item $u(\xi)\in \overline{\textnormal{conv}}(u(\partial\Omega))$
\item $\|u(\xi)\|\leq\max_{z\in\partial\Omega}\|u(z)\|$.
\end{compactenum}
\end{proposition}

\begin{proof}
It suffices to show $(i)$. Suppose that $u(x)\notin\overline{\textnormal{conv}}(u(\partial\Omega))=:M$, then by the Hahn-Banach Theorem there exists a functional $x'\in X'$ such that $\Real{\langle u(x),x'\rangle}>\sup_{m\in M}\Real{\langle m,x'\rangle}$. This contradicts the Maximum Principle for real-valued functions.
\end{proof}

\begin{remark}\label{weak remark}
The proof of Proposition \ref{weakmax} shows that it is actually enough to have $u\in \mathcal{H}(\Omega,X)\cap C_b(\Omega,X)$ such that $u(\xi)\rightharpoonup f(z)$ as $\xi\rightarrow z\in\partial\Omega$.
\end{remark}

The last tool we will need for the proof is the following density result.

\begin{lemma}\label{tensor}
Let $K$ be a compact space and $X$ be a Banach space. If $W\subset C(K,\R)$ and $Y\subset X$ are dense in the respective spaces, then the set
\begin{align*}
\left\{f\in C(K,X),f=\sum_{j=1}^{n}{f_j\otimes x_j},f_j\in W, x_j\in Y\right\}
\end{align*}
is dense in $C(K,X)$.
\end{lemma}

\begin{proof}
Let $f\in C(K,X)$ and $\varepsilon>0$. For every $k\in K$ there exists an open set $U\ni k$ such that $\|f(k)-f(l)\|\leq\varepsilon$ for all $l\in U$. Since $K$ is compact we may choose a finite number of $k_i\in K$ and $U_i$ as above such that $\bigcup_{i=1}^{n}U_i=K$. Let $(\varphi_i)_{i=1}^{n}$ be a partition of unity subordinate to the collection $(U_i)_{i=1}^{n}$. Further choose $\psi_i\in W$ such that $\|\varphi_i-\psi_i\|_\infty<\varepsilon n^{-1}\|f\|_\infty^{-1}$, $x_i\in Y$ such that $\|f(k_i)-x_i\|<\varepsilon$ and define $g:=\sum_{i=1}^{n}{\psi_i\otimes x_i}$. One easily computes that $\|f-g\|<3\varepsilon$.
\end{proof}

\begin{proof}[Proof of Theorem \ref{Perronsolution EnU}]
First suppose that $f$ is of the form $f=\sum_{n=1}^{N}{g_n\otimes x_n}$ where $g_n\in C(\partial\Omega,\R)$ and $x_n\in X$. The function $H_f:=\sum_{n=1}^{N}{H_{g_n}\otimes x_n}$ satisfies the claim. For an arbitrary function $f\in C(\partial\Omega,X)$ there exists a sequence of functions $f_n$ of the above form which converges to $f$. By the maximum principle the Perron solutions $H_{f_n}$ form a Cauchy sequence in $C_b(\Omega,X)$ and hence converge uniformly on compact sets to a harmonic function $H_f$. Let $z\in\partial\Omega$ be regular, then
\begin{align*}
&\|H_f(x)-f(z)\|\\
\leq&\|H_f-H_{f_n}\|_{\infty}+\|H_{f_n}(x)-f_n(z)\|+\|f_n(z)-f(z)\|.
\end{align*}
For large $n\in\N$ the first and last term will be smaller than a given $\varepsilon>0$, thus
\begin{align*}
\limsup_{x\rightarrow z}\|H_f(x)-f(z)\|\leq 2\varepsilon+\lim_{x\rightarrow z}\|H_{f_n}(x)-f_n(z)\|=2\varepsilon.
\end{align*}
Letting $\varepsilon\rightarrow0$ yields existence. By composing $H_f$ with an arbitrary functional $x'\in X'$ the uniqueness of $H_f$ follows from the uniqueness of the real-valued Perron solution and the Hahn-Banach Theorem.
\end{proof}

The Dirichlet problem is one of the oldest problems in partial differential equations and various different constructions of the Perron solution have been given. For the remainder of this section and during the next section we want to describe how several of these constructions work in the vector-valued case as well.

\begin{corollary}[Wiener's construction of the Perron solution]
Let $f\in C(\partial\Omega,X)$. For any continuous extension $F$ of $f$ to the whole of $\overline{\Omega}$ and every sequence $(\omega_n)$ of Dirichlet regular open subsets of $\Omega$ such that $\omega_n\ssubset\omega_{n+1}$ and $\bigcup_n\omega_n=\Omega$ we have: If $H_n$ denotes the solution of the Dirichlet problem for $\omega_n$ with boundary data $F_{|\partial\omega_n}$, then $H_n$ converges to $H_f$ uniformly on compact subsets.
\end{corollary}

\begin{proof}
The result is well known in the real-valued case, see e.g. \cite[Theorems I \& II]{KeldysDirichletproblem}. By the Maximum Principle we have that $H_n(x)\in \overline{\textnormal{conv}}(F(\partial\omega_n))\subseteq\overline{\textnormal{conv}}(F(\Omega))$. Hence $H_n$ is uniformly bounded by a constant independent of $n$. Poisson's Integral Formula implies that $(H_n)_{n\in\N}$ is equicontinuous. Further $\overline{\textnormal{conv}}(F(\Omega))$ is compact. The Arzela-Ascoli Theorem implies that each subsequence of $(H_n)_{n\in\N}$ has a subsequence $(H_{n_k})_{k\in\N}$ converging to some function $v$ on compact subsets. The real-valued case implies that $\langle H_{n_k},x'\rangle \rightarrow H_{\langle f,x'\rangle}$ and hence $v=H_f$. As the subsequence was chosen arbitrarily the claim follows.
\end{proof}

Wiener's construction is justified since an extension $F$ and an exhausting sequence of regular sets always exist, c.f. \cite[Theorem 4.1]{DugundjiVectorvaluedTietze} and \cite[Corollary 6.6.13]{ArmitageGardinerPotentialTheory}. The next construction is due to Poincar\'e and can be found in \cite[Theorem 2]{HildebrandtOnDirichletsPrinciple}. We will omit the proof since it works similar to the proof of Theorem \ref{Perronsolution EnU}.

\begin{proposition}[Poincar\'e's construction of the Perron solution]\label{poincare construction}
Let $f\in C(\partial\Omega,X)$ and let $F$ be a continuous extension of $f$ to the whole of $\overline{\Omega}$. Let $B_i$ be a sequence of non-trivial balls such that $B_i\ssubset\Omega$ and $\bigcup_i B_i=\Omega$ and let $i_n$ be a sequence such that every natural number appears infinitely often in $i_n$. For a continuous function $u$ and a ball $B$ in $\Omega$ we let $u_B$ be the function defined by $u$ outside of $B$ and the solution to the Dirichlet problem in $B$ with boundary data $u_{|\partial B}$ inside $B$. Define a sequence $u_n$ inductively via
\begin{align*}
u_0=F\textnormal{ and }u_n=(u_{n-1})_{B_{i_n}}.
\end{align*}
Then the sequence $u_n$ converges to $H_f$ uniformly on compact subsets.
\end{proposition}

Another possible way of constructing the Perron solution is via \emph{harmonic measures}. We summarize the construction of these measures and refer to \cite[Chapter 6.4]{ArmitageGardinerPotentialTheory} for more information. For every $\xi\in\Omega$ the mapping
\begin{align*}
C(\partial\Omega,\R)&\rightarrow\R\\
f&\mapsto H_f(\xi)
\end{align*}
is well-defined, linear and positive. It follows from the Riesz Representation Theorem that there exists a unique probability measure $\mu_\xi$ on the Borel algebra $\mathcal{B}(\partial\Omega)$ such that
\begin{align}\label{rieszrep}
H_f(\xi)=\int_{\partial\Omega}{f\,d\mu_\xi}.
\end{align}

\begin{corollary}[Construction via harmonic measures]
Let $f\in C(\partial\Omega,X)$. The function $H_f$ can be defined via \eqref{rieszrep}.
\end{corollary}

\begin{proof}
Let $x'\in X'$. Since $x'$ commutes with integration the claim holds due to the uniqueness of the Perron solution.
\end{proof}

\begin{remark}
Another way to construct the Perron solution is to use Hilbert space theory in the space $H^1(\Omega,\R)$. This construction can be found in \cite[Theorem 1]{HildebrandtOnDirichletsPrinciple}, \cite[II,\S7, 4.]{DautrayLionsMathematicalAnalysis} or \cite{ArendtDanersDirichletProblem}. Obviously, this approach fails if $X$ is not a Hilbert space. However, if $X$ happens to be a Hilbert space, the Perron solution can be constructed via Hilbert space methods as in the real-valued case with minor changes and we will omit to carry out the proof.
\end{remark}

\section{Perron's Method on Banach Lattices}
The classical method to obtain the Perron solution is to construct it as the pointwise supremum of subsolutions, see e.g. c.f. \cite[Chapter 6.6]{ArmitageGardinerPotentialTheory}, \cite[Chapter 2]{HelmsPotentialTheory} or \cite[Theorem II]{KeldysDirichletproblem}. This crucially depends on the order of $\R$ and hence makes no sense in general Banach spaces. Thus throughout this section let $X$ be a Banach lattice. We will use the partial ordering on $X$ to generalize Perron's approach.\\

As in the real-valued case a function $v\in C(\Omega,X)$ is called \emph{subharmonic} if for all $\xi\in\Omega$ there exists $R>0$ such that $B(\xi,R)\ssubset\Omega$ and we have
\begin{align*}
v(\xi)\leq\avgint_{\partial B(\xi,r)}{v\,d\sigma_{d-1}}
\end{align*}
for all $0<r<R$. Testing with positive functionals and using the real-valued case \cite[Chapter 11, Exercise 5]{AxlerBourdonRamneyHarmonicFunction} reveals that this inequality actually holds for all $0<r<\dist(\xi,\partial\Omega)$. Let $f\in C(\partial\Omega,X)$. Then a \emph{continuous subsolution} of the Dirichlet problem with boundary data $f$ is a function $v\in C(\overline{\Omega},X)$ that is subharmonic and satisfies $v(z)\leq f(z)$ for all $z\in\partial\Omega$. Analogously one defines superharmonic functions and continuous supersolutions. The sets of continuous sub-\slash supersolutions will be denoted by $\mathcal{CS}^-_f$ and $\mathcal{CS}^+_f$ respectively. 

\begin{remark}
The above definition is common in books on partial differential equations, such as \cite{GilbargTrudingerEllipticpde} and \cite{AxlerBourdonRamneyHarmonicFunction}, where they are just called \glqq subsolutions\grqq. In potential theory subsolutions are defined in a weaker sense, see , \cite{ArmitageGardinerPotentialTheory} or \cite{HelmsPotentialTheory}. For this reason we chose the name \glqq continuous subsolution\grqq.
\end{remark}

If $X=\R$ then $f\in C(\partial\Omega,\R)$ is automatically bounded and hence $v_-\equiv\min{f}$ and $v_+\equiv\max{f}$ are continuous sub-\slash supersolutions. In an arbitrary Banach lattice, $v_-$ and $v_+$ can only exist if $f$ is \emph{bounded in order}, i.e. there exist $s,S\in X$ such that $s\leq f(\xi)\leq S$. The set of all order bounded and continuous functions will be denoted by $C_{\textnormal{ob}}(\partial\Omega,X)$. We give a counterexample.
\begin{example}
Let $(z_n)_{n\in\N}$ be a pairwise distinct sequence of points in $\partial\Omega$ converging to $z$. Set $f(z_n):=\frac{e_n}{n}\in\ell^1$ and $f(z):=0$, then $f$ is continuous. Using Dugundi's Extension Theorem \cite[Theorem 4.1]{DugundjiVectorvaluedTietze} we may extend $f$ to a function $f\in C(\partial\Omega,\ell^1)$ that is not bounded from above in the order of $\ell^1$. Hence there is no guarantee that $\mathcal{CS}_f^+\neq\emptyset$.
\end{example}
For $f\in C_{\textnormal{ob}}(\partial\Omega,X)$ with lower bound $s$ and upper bound $S$ the sets $\mathcal{CS}_f^-$ and $\mathcal{CS}_f^+$ are nonempty as they contain the functions $v_-\equiv s$ and $v_+\equiv S$ respectively.

\begin{proposition}[Maximum Principle for lattices]
Let $u,v\in C(\overline{\Omega},X)$ and let $v$ be subharmonic and $u$ harmonic. If $v\leq u$ on $\partial\Omega$ then $v\leq u$ in $\overline{\Omega}$.
\end{proposition}

\begin{proof}
Let $x'\in X'_+$ then $\langle v,x'\rangle$ and $\langle u,x'\rangle$ are real-valued subharmonic respectively harmonic functions satisfying $\langle v,x'\rangle\leq \langle u,x'\rangle$ on $\partial\Omega$. By the real-valued maximum principle we have that $\langle v,x'\rangle\leq\langle u,x'\rangle$ in $\Omega$. Since $X'_+$ determines positivity the claim follows immediately.
\end{proof}

Let $f\in C(\partial\Omega,X)$ and suppose that $\mathcal{CS}_f^-$ and $\mathcal{CS}_f^+$ are non-empty. Suppose further that the pointwise supremum $\sup_{v\in \mathcal{CS}_f^-}v(\xi)$ and the pointwise infimum $\inf_{v\in \mathcal{CS}_f^+}v(\xi)$ exist for all $\xi\in\Omega$. Then we will denote them by $\underline{H}_f$ and $\overline{H}_f$. The following obvious result is the motivation for Perron's method.

\begin{proposition}\label{obvious}
Suppose that the Dirichlet problem with boundary data $f\in C(\partial\Omega,X)$ admits a classical solution $H_f$, then $\underline{H}_f=\overline{H}_f=H_f$.
\end{proposition}

\begin{proof}
Since $H_f$ is harmonic and takes on the prescribed boundary values it follows that $H_f\in \mathcal{CS}_f^-\cap \mathcal{CS}_f^+$. The Maximum Principle implies that $H_f=\sup{\mathcal{CS}_f^-}=\inf{\mathcal{CS}_f^+}$.
\end{proof}

In the real-valued case we have

\begin{theorem}\label{Perronmethodreal}
Let $f\in C(\partial\Omega,\R)$, then $\underline{H}_f=\overline{H}_f=H_f$.
\end{theorem}

We will give a proof of this theorem later on. Even if the boundary data is not order bounded it might still be true that $\mathcal{CS}_f^\pm\neq\emptyset$. For example this is the case if the domain is Dirichlet regular. The following example covers a variety of domains which still allow Perron's method to work. It allows us to consider such common examples of irregular sets as the punctured disk $B_\C(0,1)\backslash\{0\}$.

\begin{example}\label{subsolutions possible domain}
Let $\Omega^\ast$ be a Dirichlet regular domain and let $z_1,\ldots,z_k\in\Omega^\ast$. If $\Omega:=\Omega^\ast\backslash\{z_1,\ldots,z_k\}$ then for every $f\in C(\partial\Omega,X)$ the sets $\mathcal{CS}_f^\pm$ are non-empty. To see this note that $\{z_1,\ldots,z_k\}$ is the set of irregular points of $\Omega$. Thus the uniqueness of the Perron solution implies that $H_f=H_{f^\ast}$ for $f^\ast:=f_{|\partial\Omega^\ast}$. Hence $H_f(z_j)$ is well defined and $H_f\in C(\overline{\Omega},X)$. We define
\begin{align*}
v^\pm(\xi):=H_f(\xi)\pm\sum_{j=1}^{k}{|H_f(z_j)-f(z_j)|},
\end{align*}
then it is easy to see that $v^\pm\in \mathcal{CS}_f^\pm$. The construction shows that we may also withdraw an infinite number of points in a regular domain as long as every connected component contains only a finite number of these points.
\end{example}

In some cases, the Banach lattice itself has the property that sub- and supersolutions exist.

\begin{example}\label{subsolutions possible AM}
Let $X$ be an AM-space, see e.g. \cite[Definition 4.20]{AliprantisBurkishawPositiveOperators}, then every compact set is order bounded \cite[Theorem 4.30]{AliprantisBurkishawPositiveOperators}. Hence for every continuous boundary data $f$ the sets $\mathcal{CS}_f^\pm$ are non-empty.
\end{example}

Let $u\in C(\Omega,X)$ be subharmonic and let $B:=B(x,r)\ssubset\Omega$. We define the \emph{harmonic lifting} via
\begin{align*}
u_B:=
\begin{cases}
u&\textnormal{ on }\Omega\backslash B\\
H_{u_{|\partial B}}&\textnormal{ on } B.
\end{cases}
\end{align*}
The Maximum Principle implies that $u_B\geq u$ and it follows that $u_B$ is still subharmonic. Further if $v\in C(\Omega,X)$ is another subharmonic function, then the function $u\vee v$ is subharmonic as well. Let $\mathcal{F}\subset C(\Omega,X)$ be a set of subharmonic functions. We say that $\mathcal{F}$ is a \emph{Perron family} if
\begin{compactenum}[\upshape (a)]
\item $\mathcal{F}$ is upwards directed (a sufficient condition would be that $\mathcal{F}$ is closed under taking the pointwise supremum of two functions)
\item $u_B\in\mathcal{F}$ for every $u\in\mathcal{F}$ and every ball $B:=B(x,r)\ssubset\Omega$.
\end{compactenum} 

A Banach lattice $X$ is called \emph{order complete} if every order bounded subset of $X$ has a supremum and an infimum. Common examples are $L^p$-spaces $(p<\infty)$, a counterexample is the space $C([0,1])$. A functional $x'\in X'_+$ is called \emph{order continuous} if for every upwards directed set $A\subset X$ which has a supremum, we have $\langle\sup A,x'\rangle=\sup\langle A,x'\rangle$.

\begin{proposition}\label{Perron family weak}
Let $X$ be a Banach lattice which is order complete and assume that the order continuous functionals separate. Suppose that $\mathcal{F}\subset C(\Omega,X)$ is a Perron family which is pointwise order bounded and locally bounded. Then the function
\begin{align*}
\xi\mapsto\sup_{v\in\mathcal{F}}v(\xi)
\end{align*}
is well defined and harmonic.
\end{proposition}

\begin{proof}
The case $X=\R$ is well known, see \cite[Theorem 2.12]{GilbargTrudingerEllipticpde}. In general, the supremum exists since $X$ is order complete. Now let $x'\in X'_+$ be order continuous, then $\langle\mathcal{F},x'\rangle$ is a Perron family as well and we have that $\langle\sup_{v\in\mathcal{F}}v,x'\rangle=\sup_{v\in\mathcal{F}}\langle v,x'\rangle$. The case $X=\R$ shows that the latter is harmonic. By Theorem \ref{Arendt}(\ref{Arendt very weak}) it follows that $\sup_{v\in\mathcal{F}}v$ is harmonic.
\end{proof}

\begin{example}\label{order continuous examples}
\begin{compactenum}[(a)]
\item If $X$ has order continuous norm, see \cite[Section 4.1]{AliprantisBurkishawPositiveOperators}, then $X$ automatically satisfies the conditions of Proposition \ref{Perron family weak}. Examples are reflexive Banach lattices and also $L^1$.
\item If $X$ is a Banach lattice, then $X'$ satisfies the conditions of Proposition \ref{Perron family weak}. To see this let $(x'_\alpha)$ be upwards directed in $X'$. We define $x':X_+\rightarrow\R$ via $x\mapsto\sup_{\alpha}\langle x,x'_\alpha\rangle$. Immediately $x'$ is sublinear on $X_+$. But for every $\alpha,\beta$ there exists $\gamma$ such that $x'_\alpha,x'_\beta\leq x'_\gamma$, from which we obtain that $\langle x+y,x'_\gamma\rangle\geq\langle x,x'_\alpha\rangle+\langle y,x'_\beta\rangle$ for every $x,y\in X_+$. Taking the suprema on both sides we obtain that $x'$ is also superlinear and hence linear on $X_+$. Extend $x'$ to all of $X$ via $x\mapsto\langle x_+,x'\rangle-\langle x_-,x'\rangle$ to obtain a linear functional. From the definition of $x'$ one immediately has that $x'=\sup_\alpha x'_\alpha$. Further one has that every $x\in X_+\subset X''$ is a positive order continuous functional and $X$ separates $X'$ naturally. To show the order completeness of $X'$ let $A\subset X'$ be any subset which has an upper bound. Let $\tilde{A}$ be the set of all finite suprema of elements in $A$, then $\tilde{A}$ is upwards directed. It follows from the first considereation that $\sup\tilde{A}$ exists and one immediately sees that $\sup A=\sup\tilde{A}$.
\end{compactenum}
\end{example}

If $f\in C(\partial\Omega,X)$ such that $\mathcal{CS}_f^-$ is nonempty, then $\mathcal{CS}_f^-$ and $-\mathcal{CS}_f^+$ are Perron families which are bounded from above. So we obtain

\begin{theorem}\label{Perron first step}
Suppose that $X$ is order complete and that the order continuous functionals separate. Let $f\in C(\partial\Omega,X)$ such that $\mathcal{CS}_f^\pm\neq\emptyset$. Then $\underline{H}_f$ and $\overline{H}_f$ exist and are harmonic.
\end{theorem}

Suppose that $\mathcal{CS}_f^\pm\neq\emptyset$ for every boundary data $f\in C(\partial\Omega,X)$ and that $X$ is as in Theorem \ref{Perron first step}. The question we want to study now is whether in this situation $\underline{H}_f=\overline{H}_f=H_f$. Consider the operators $T^-f\mapsto\underline{H}_f$ and $T^+f\mapsto \overline{H}_f$. It is easy to see that they are super-\slash sublinear. By what we have seen so far, they map into the bounded harmonic functions $\mathcal{H}_b(\Omega,X)$ and satisfy $T^-f=T^+f=H_f$ if a classical solution of the Dirichlet problem with boundary data $f$ exists. This suggests an operator theoretic approach. The operators $T^-$ and $T^+$ are monotone, i.e. if $f\leq g$, then $T^-f\leq T^-g$. For linear operators we have

\begin{theorem}[Keldysh]
Let $T$ be an operator that maps $C(\partial\Omega,\R)$ into the bounded harmonic functions on $\Omega$ such that $T$
\begin{compactenum}[\upshape (i)]
\item is linear
\item positive
\item and maps $f$ to the classical solution if it exists.
\end{compactenum}
Then $Tf=H_f$ for all $f\in C(\partial\Omega,\R)$
\end{theorem}

\begin{proof}
See \cite[Theorem 4.11]{LandkofPotentialtheory}.
\end{proof}

\begin{remark}
The original proof is available in Russian only. For a reference, see \cite{LandkofPotentialtheory}. Landkof's proof differs from the original one. A simplified proof using Keldysh's ideas was given by Brelot \cite{BrelotKeldyshtheorem} in French. We will reproduce this proof in Theorem \ref{Keldysh for OBC}.
\end{remark}

\begin{corollary}[Keldysh's Theorem for lattices]
Let $X$ be a Banach lattice and let $T:C(\partial\Omega,X)\rightarrow\mathcal{H}_b(\Omega,X)$ such that $T$ is
\begin{compactenum}[\upshape (i)]
\item is linear
\item positive
\item and maps $f$ to the classical solution if it exists.
\end{compactenum}
Then $Tf=H_f$ for all $f\in C(\partial\Omega,X)$
\end{corollary}

\begin{proof}
We first show that for $x\in X_+$ and $f\in C(\partial\Omega,\R)$ we have $T(f\otimes x)=H_{f\otimes x}$. Since $H_{f\otimes x}=H_f\otimes x$ it is enough to show that $\langle T(f\otimes x),x'\rangle=H_f$ for each $x'\in X'_+$ such that $\langle x,x'\rangle=1$. Choose such $x'$ and define the operator
\begin{align*}
S:C(\partial\Omega,\R)&\rightarrow\mathcal{H}(\Omega,\R)\\
f&\mapsto\langle T(f\otimes x),x'\rangle.
\end{align*}
It is immediately clear that $S$ satisfies the conditions of Keldysh's Theorem and hence
\begin{align*}
\langle T(f\otimes x),x'\rangle=H_f.
\end{align*}
Next we may consider arbitrary $x=x^+-x^-$ by linearity and further obtain $Tf=H_f$ for any function of the form $f=\sum_{n=1}^{N}{f_n\otimes x_n}$. Now let $f\in C_{\textnormal{ob}}(\partial\Omega,X)$ be arbitrary, then there exist functions $f_n$ of the above form such that $f_n\rightarrow f$. Since $T$ is positive, linear and its range is contained in the Banach lattice $C_b(\Omega,X)$ it is also continuous and hence the claim follows for $f$.
\end{proof}

Since $T^+$ and $T^-$ are not linear, we cannot immediately apply Keldysh's Theorem and need to show some further results.  An ordered vector space $W$ is said to have the \emph{interpolation property} if for every two sets $A,B\subset W$ such that $a\leq b$ for all $a\in A$, $b\in B$ there exists $w\in W$ such that $a\leq w\leq b$ for all $a\in A$, $b\in B$. A function $p$ mapping a vector space $V$ into $W$ is called \emph{sublinear} if
\begin{align*}
p(\lambda v)=\lambda p(v)\textnormal{ and } p(v_1+v_2)\leq p(v_1)+p(v_2)
\end{align*}
for all $\lambda\geq0$ and $v,v_1,v_2\in V$.

\begin{theorem}[Hahn-Banach-Kantorowicz]
Let $V$ be a vector space and $U\subset V$ a subspace. Let $W$ be and ordered vector space which has the interpolation property. Let
\begin{align*}
\varphi:U\rightarrow W
\end{align*}
be linear and
\begin{align*}
p:V\rightarrow W
\end{align*}
be sublinear such that $\varphi\leq p$ on $U$. Then $\varphi$ can be extended to $V$ such that $\varphi\leq p$ on $V$.
\end{theorem}

\begin{proof}
Using the interpolation property, this can be established analogously to the real-valued result, see \cite[Theorem 1.1]{BrezisFunctionalanalysis}.
\end{proof}

\begin{corollary}
Given $p$ as in the Hahn-Banach-Kantorowicz Theorem, we have
\begin{align*}
p(v)=\max\{\varphi(v),\varphi:V\rightarrow W\textnormal{ linear},\varphi\leq p\}
\end{align*}
for every $v\in V$.
\end{corollary}

\begin{proof}
For $v\in V$ define $\varphi(\lambda v):=\lambda p(v)$ for all $\lambda\in\R$ and extend $\varphi$ to $V$ via the Hahn-Banach-Kantorowicz Theorem. The claim follows immediately.
\end{proof}

We show that one can apply the above considerations to the case where $W$ is the space of harmonic and bounded functions.

\begin{lemma}\label{interpolation property lemma}
Let $A,B\subset\mathcal{H}_b(\Omega,X)$ such that $a\leq b$ for all $a\in A$, $b\in B$. If $X$ is order complete and the order continuous functionals separate, then there exists a function $h\in\mathcal{H}_b(\Omega,X)$ such that $a\leq h\leq b$ holds for all $a\in A$, $b\in B$.
\end{lemma}

\begin{proof}
Consider the set
\begin{align*}
A^*:=\{a\in C_b(\Omega,X), a\textnormal{ subharmonic}, a\leq b \textnormal{ for all }b\in B\}\supset A.
\end{align*}
One easily sees that $A^*$ is a Perron family and hence has a harmonic supremum $h$ by Proposition \ref{Perron family weak}. The function $h$ satisfies the claim.
\end{proof}

\begin{remark}
It follows from the preceding lemma, that $\mathcal{H}_b(\Omega,X)$ is an order complete vector lattice. We can also show that $\mathcal{H}_b(\Omega,\R)$ is even a Banach lattice. Since this is not necessary for our considerations, we have deferred this discussion to the Appendix \ref{harmonic lattice appendix}.
\end{remark}

\begin{proposition}
Let $p:C(\partial\Omega,X)\rightarrow\mathcal{H}_b(\Omega,X)$ such that
\begin{compactenum}[\upshape (i)]
\item $p$ is sublinear
\item $p$ is monotone, i.e. $p(f)\leq p(g)$ if $f\leq g$
\item $p$ maps $f$ to the classical solution of the Dirichlet problem if it exists.
\end{compactenum}
Then $p(f)=H_f$ for all $f\in C(\partial\Omega,X)$
\end{proposition}

\begin{proof}
Consider the set
\begin{align*}
\mathcal{M}:=\{\varphi:C(\partial\Omega,X)\rightarrow \mathcal{H}_b(\Omega,X)\textnormal{ linear },\varphi\leq p\}.
\end{align*}
We will show that $\mathcal{M}$ actually consists of only one element, namely $f\mapsto H_f$. The claim then follows since $p(f)=\max_{\varphi\in\mathcal{M}}\varphi(f)$. Let $\varphi\in\mathcal{M}$. Then $\varphi$ is positive. In fact, let $f\leq0$. Then $p(f)\leq p(0)=0$ and hence $\varphi(f)\leq0$. Now let $f\in C(\partial\Omega,X)$ such that $H_f$ is a classical solution. Then $\varphi(f)\leq H_f$. On the other hand $-\varphi(f)=\varphi(-f)\leq p(-f)=H_{-f}=-H_f$. Hence $\varphi(f)=H_f$. In total, $\varphi$ satisfies the conditions of Keldysh's Theorem which finishes the proof.
\end{proof}

\begin{theorem}\label{X Perron final}
Assume that $X$ is order complete, the order continuous functionals separate and that $\mathcal{CS}_f^\pm\neq\emptyset$ for every $f\in C(\partial\Omega,X)$. Then $\underline{H}_f=\overline{H}_f=H_f$ for all $f\in C(\partial\Omega,X)$.
\end{theorem}

\begin{proof}
The mapping $T^+$ satisfies the assumptions of the preceding proposition. The reasoning for $T^-$ is analogous.
\end{proof}

Examples \ref{subsolutions possible domain}, \ref{subsolutions possible AM} and \ref{order continuous examples} describe situations in which the conditions of Theorem \ref{X Perron final} are fulfilled. We also obtain the

\begin{proof}[Proof of Theorem \ref{Perronmethodreal}]
If $X=\R$, the assumptions of Theorem \ref{X Perron final} are satisfied.
\end{proof}

\begin{openproblem}
Is there a function $f\in C(\partial\Omega,X)$ for some $\Omega\subset\R^d$ open and bounded and some Banach lattice $X$ such that $\mathcal{CS}_f^\pm=\emptyset$?
\end{openproblem}

A negative answer to the problem above would imply that Perron's method works whenever the Banach lattice is order complete and has the property that the order continuous functionals separate the space. To avoid this open question we will now demand a stronger continuity assumption of the boundary values and reevaluate the operator theoretic approach of Theorem \ref{X Perron final}. Let $X$ be a Banach lattice which is order complete and whose order continuous functionals separate. A function $f:\partial\Omega\rightarrow X$ is called \emph{order boundedly continuous} at $z_0\in\partial\Omega$ if for every $\varepsilon>0$ there exists $y\in X_+$ with $\|y\|\leq\varepsilon$ and $\delta>0$ such that
\begin{align*}
|f(z)-f(z_0)|\leq y
\end{align*}
for all $z\in\partial\Omega$ for which $|z-z_0|\leq\delta$. The set of all order boundedly continuous functions on $\partial\Omega$ will be denoted by $OBC(\partial\Omega,X)$.

\begin{example}
If $X$ is an AM-space, then every continuous function $f\in C(\partial\Omega,X)$ is automatically an element of $OBC(\partial\Omega,X)$. Indeed, for $z_0\in\partial\Omega$ and $\varepsilon>0$ let $\delta>0$ such that $\|f(z)-f(z_0)\|<\varepsilon$ whenever $|z-z_0|<\delta$. By \cite[Theorem 4.30]{AliprantisBurkishawPositiveOperators} there exists $y:=\sup\{|f(z)-f(z_0)|,|z-z_0|<\delta\}$ and the proof of this theorem can easily be modified to show that $\|y\|<\varepsilon$. Hence $f$ is order boundedly continuous at $z_0$.
\end{example}

A simple compactness argument shows that every $f\in OBC(\partial\Omega,X)$ is order bounded and hence the sets $\mathcal{CS}_f^+$ and $\mathcal{CS}_f^-$ are nonempty. It follows from Proposition \ref{Perron first step} that the operators $\overline{H}_f$ and $\underline{H}_f$ are well defined. We now prove Keldysh's Theorem in this case.

\begin{theorem}[Keldysh's Theorem for $OBC$-functions]\label{Keldysh for OBC}
Let $T:OBC(\partial\Omega,X)\rightarrow \mathcal{H}_b(\Omega,X)$ such that
\begin{compactenum}[(a)]
\item $T$ is linear
\item $T$ is positive
\item $Tf=H_f$ if $H_f$ is a classical solution.
\end{compactenum}
Then $Tf=H_f$ for all $f\in OBC(\partial\Omega,X)$.
\end{theorem}

\begin{proof}
Let $f\in OBC(\partial\Omega,X)$, $z_0\in\partial\Omega$ regular, $\varepsilon>0$. There exists a $\delta>0$ and a $y\in X_+$ such that $|f(z)-f(z_0)|\leq y$ if $z\in\partial\Omega$ such that $|z-z_0|\leq\delta$. By \cite[Lemme 1]{BrelotKeldyshtheorem} there exists a harmonic function $w$ in $\Omega$ which is non negative, continuous up to the boundary of $\Omega$ and satisfies $w(z_0)<\varepsilon$ as well as $w(z)\geq1$ for all $z\in\partial\Omega$ such that $|z-z_0|>\delta$. Let $s:=2\sup_{\partial\Omega}|f|$, then we have that
\begin{align*}
f\leq f(z_0)+y+w\otimes s.
\end{align*}
Indeed if $|z-z_0|\leq\delta$, then this follows from the choice of $\delta$ and the positivity of $w\otimes s$. Otherwise it follows from the choice of $s$, the properties of $w$ and the positivity of $y$. Using the properties of $T$ we obtain
\begin{align*}
Tf\leq f(z_0)+y+w\otimes s
\end{align*}
and hence for every $x'\in X'_+$ we have
\begin{align*}
\limsup_{\substack{\xi\rightarrow z_0 \\ \xi\in\Omega}}\langle Tf(\xi),x'\rangle&\leq\langle f(z_0),x'\rangle+\varepsilon\|x'\|+\lim_{\substack{\xi\rightarrow z_0 \\ \xi\in\Omega}}w(\xi)\langle s,x'\rangle\\
&\leq\langle f(z_0),x'\rangle+\varepsilon\|x'\|+\varepsilon\langle s,x'\rangle.
\end{align*}
Since $\varepsilon>0$ was arbitrary, it follows that
\begin{align*}
\limsup_{\substack{\xi\rightarrow z \\ \xi\in\Omega}}\langle Tf(\xi),x'\rangle\leq \langle f(z_0),x'\rangle
\end{align*}
and analogously
\begin{align*}
\liminf_{\substack{\xi\rightarrow z \\ \xi\in\Omega}}\langle Tf(\xi),x'\rangle\geq \langle f(z_0),x' \rangle.
\end{align*}
These two inequalities show that $\langle Tf(\xi),x'\rangle\rightarrow\langle f(z_0),x'\rangle$ and by the uniqueness of the Perron solution this implies that $Tf=H_f$.
\end{proof}

From here on we can work analogously to the proof of Theorem \ref{X Perron final} and obtain

\begin{theorem}
Let $X$ be a Banach lattice which is order complete such that the order continuous functional separate. Then $\underline{H}_f=\overline{H}_f=H_f$ for every $f\in OBC(\partial\Omega,X)$.
\end{theorem}

For the rest of this section, we want to consider Perron's method on a fixed function rather than the operator theoretic approach above where we had to consider all functions on the boundary at once. We will only consider order bounded functions on special Banach lattices and reconstruct special subsolutions from the real-valued case. As a result we will have

\begin{theorem}\label{special cases}
Let $X$ be an AM-space with unit (see \cite[p. 195]{AliprantisBurkishawPositiveOperators}) or $X=\ell^p$ where $1\leq p\leq\infty$. Then for every $f\in C_{\textnormal{ob}}(\partial\Omega,X)$ we have that $\underline{H}_f=\overline{H}_f=H_f$.
\end{theorem}

Note that in the case of an AM-space every bounded function is already order bounded, see Example \ref{subsolutions possible AM}. But Theorem \ref{X Perron final} does not apply since $X$ is not order complete in general, e.g. take $X=C([0,1])$.

\begin{lemma}\label{reconstruction enough}
Let $X$ be a Banach lattice and $f\in C(\partial\Omega,X)$. Suppose that $W\subset X'_+$ determines positivity and that for all $x'\in W,\varepsilon>0,v\in \mathcal{CS}_{\langle f,x'\rangle}^-,\xi\in\Omega$ there exists a function $\overline{v}\in \mathcal{CS}_f^-$ satisfying $\langle\overline{v}(\xi),x'\rangle\geq v(\xi)-\varepsilon$. Then $\underline{H}_f=u_f$. The analogous result for $\overline{H}_f$ holds.\\
\end{lemma}

\begin{proof}
By symmetry it is enough to show the claims for $\underline{H}_f$.\\
Let $\xi\in\Omega$, $\varepsilon\geq0$. By Theorem \ref{Perronmethodreal} we may choose $v\in \mathcal{CS}_{\langle f,x'\rangle}^-$ such that
\begin{align*}
v(\xi)\geq H_{\langle f,x'\rangle}(\xi)-\varepsilon=\langle H_f(\xi),x'\rangle-\varepsilon.
\end{align*}
Then the function $\overline{v}$ corresponding to $v$ satisfies $\langle \overline{v}(\xi),x'\rangle\geq\langle H_f(\xi),x'\rangle-2\varepsilon$. By the Regular Maximum Principle we have that $H_f(\xi)\geq w(\xi)$ for every $w\in \mathcal{CS}_f^-$. Let $s$ be any upper bound of  $\{w(\xi),w\in \mathcal{CS}_f^-\}$, then $\langle s,x'\rangle\geq\langle\overline{v}(\xi),x'\rangle\geq\langle H_f(\xi),x'\rangle-2\varepsilon$. Since $\varepsilon>0$ and $x'\in W$ were arbitrary we obtain that $s\geq H_f(\xi)$.
\end{proof}

\begin{proof}[Proof of Theorem \ref{special cases}]
We will show that in both cases we can construct subsolutions as in Lemma \ref{reconstruction enough}. First let $X$ be an AM-space with unit. By Kakutani's Theorem \cite[Theorem 4.29]{AliprantisBurkishawPositiveOperators} we may assume that $X=C(K)$ for some compact topological space $K$. Replacing $f$ by $f+\|f\|_\infty$ we may assume that $f\geq0$. The family $(\delta_a)_{a\in K}$ determines positivity. Let $a\in K$, $\varepsilon>0$ and $v\in \mathcal{CS}_{f(\cdot)(a)}^-$. For every $z\in\partial\Omega$ we have that $v(z)-\varepsilon<f(z)(a)$. Since the function $(z,k)\mapsto f(z)(k)$ is continuous and $\partial\Omega$ is compact there exists a neighbourhood $U^*$ of $a$ such that $f(z)(k)>v(z)-\varepsilon$ for all $z\in\partial\Omega$ and all $k\in U^*$. By Tietze's Extension Theorem there exists a function $g\in C(K)$ such that
\begin{align*}
g(k)
\begin{cases}
=1,&\textnormal{if }k=a\\
=0,&\textnormal{if }k\notin U^*\\
\in[0,1],&\textnormal{otherwise}
\end{cases}.
\end{align*}
Then the function $\xi\mapsto\overline{v}(\xi)(k):=(v(\xi)-\varepsilon)g(k)$ satisfies the demands.\\

Now let $X=\ell^p$ and let $s$ be an upper bound of $f\in C_\textnormal{ob}(\partial\Omega,X)$. We choose the coordinate functionals $e_n$ as separating subset of $X'$. Let $f_n:=\langle f,e_n\rangle$ and let $v\in \mathcal{CS}_f^+$. Set
\begin{align*}
\overline{v}(\xi)_k:=
\begin{cases}
v\textnormal{ if }k=n\\
s_k \textnormal{ if }k\neq n.
\end{cases}
\end{align*}
One easily sees that $\overline{v}$ satisfies the demands.
\end{proof}

\section{Application: The Poisson Problem for Elliptic Operators}

Again, let $X$ be a Banach space and $\Omega\subset\R^d$ be open and bounded. Recall that a function $u\in C(\Omega,X)$ is \emph{uniformly H\"older continuous} if there exist constants $C>0$ and $0<\alpha<1$ such that
\begin{align*}
\|u(\xi)-u(\eta)\|\leq C|\xi-\eta|^\alpha
\end{align*}
holds for all $\xi,\eta\in\Omega$. The set of all such mappings will be denoted by $C^\alpha(\overline{\Omega},X)$. A function $u\in C(\Omega,X)$ will be called \emph{locally H\"older continuous} if for every $\omega\ssubset\Omega$ there exists $\alpha_\omega$ such that $u\in C^{\alpha_\omega}(\omega,X)$. Further we denote by $C^{2,\alpha}(\overline{\Omega},X)$ the set of all functions $u\in C^2(\overline{\Omega},X)$ such that the second derivatives of $u$ are elements of $C^\alpha(\overline{\Omega},X)$. Equipped with the norms
\begin{align*}
\|u\|_\alpha:=\|u\|_{C(\overline{\Omega},X)}+\sup_{\xi,\eta\in\Omega}\frac{\|u(\xi)-u(\eta)\|}{|\xi-\eta|^\alpha}
\end{align*}
respectively
\begin{align*}
\|u\|_{2,\alpha}:=\|u\|_{C^2(\overline{\Omega},X)}+\sum_{|\beta|=2}{\sup_{\xi,\eta\in\Omega}\frac{\|D^\beta u(\xi)-D^\beta u(\eta)\|}{|\xi-\eta|^\alpha}}
\end{align*}
the spaces $C^\alpha(\overline{\Omega},X)$ and $C^{2,\alpha}(\overline{\Omega},X)$ are Banach spaces. In this section we consider a second order differential operator
\begin{align*}
L:=\sum_{i,j=1}^{d}{a_{ij}(\cdot)D_{ij}}+\sum_{i=1}^{d}{b_i(\cdot)D_i}+c(\cdot)
\end{align*}
whose coefficients $a_{ij},b_i$ and $c$ are in $C^\alpha(\overline{\Omega},\R)$. Further assume that $c\leq0$ and that the $a_{ij}$ are symmetric and satisfy a strict ellipticity condition
\begin{align*}
\sum_{i,j=1}^{d}{\xi_ia_{ij}(x)\xi_j}\geq\lambda|\xi|^2
\end{align*}
for some $\lambda>0$. For $f\in C(\partial\Omega,X)$ and $g\in C(\Omega,X)$ consider the \emph{Poisson problem for $L$} given by
\begin{align*}
\begin{cases}
Lu=g&\textnormal{in }\Omega\\
u=f&\textnormal{on }\partial\Omega
\end{cases}
\end{align*}
A \emph{classical solution} to this problem is a function $u\in C^2(\Omega,X)\cap C(\overline{\Omega},X)$ which satisfies the above. We want to apply the results of the previous sections to show
\begin{theorem}\label{PoissonSolutionL}
Suppose that $\Omega$ has a $C^{2,\alpha}$-boundary and let $L$ be as above. Assume that $g\in C^\alpha(\overline{\Omega},X)$ and $f=F_{|\partial\Omega}$, where $F\in C^{2,\alpha}(\overline{\Omega},X)$. Then there exists a unique classical solution $u$ to the Poisson problem. Moreover $u\in C^{2,\alpha}(\overline{\Omega},X)$.
\end{theorem}

Before we can show this we need to consider the special case $L=\Delta$. The \emph{fundamental solution} for the (real-valued) Dirichlet problem is given by
\begin{align*}
\Gamma(x-y)=
\begin{cases}
\frac{1}{2\pi}\log|x-y|,\ &d=2\\
\frac{1}{d(2-d)|B(0,1)|}|x-y|^{2-d},&d>2.
\end{cases}
\end{align*}
The \emph{Newtonian Potential} of $g\in L^\infty(\Omega,X)$ is teh function defined via $w:=\Gamma\ast g$. Analogously to the real-valued case, see e.g. \cite[Lemmas 4.1 \& 4.2]{GilbargTrudingerEllipticpde}, we obtain

\begin{lemma}
Let $\Omega\subset\R^d$ be open and bounded. Let $g\in C(\Omega,X)$ be bounded and locally H{\"o}lder continuous and let $w$ be its Newtonian potential. Then $w\in C^2(\Omega,X)\cap C(\overline{\Omega},X)$ with $\Delta w=g$.
\end{lemma}

Assume now that $g$ is bounded and locally H{\"o}lder continuous. Using this lemma we immediately see that the Poisson Problem for $L=\Delta$ is equivalent to the problem
\begin{align*}
\begin{cases}
\Delta v=0&\textnormal{in }\Omega\\
v=f-w&\textnormal{on }\partial\Omega
\end{cases}
\end{align*}
where $v:=u-w$. Taking $v$ as the Perron solution of the above Dirichlet Problem and adding $w$ we obtain

\begin{theorem}\label{PoissonSolutionLaplace}
Let $\Omega\subset\R^d$ be open and bounded. For every bounded and locally H{\"o}lder continuous function $g\in C(\Omega,X)$ and $f\in C(\partial\Omega,X)$ there exists a unique function $u_{f,g}\in C^2(\Omega,X)$ satisfying
\begin{align*}
&\Delta u_{f,g}=g\\
&\lim_{x\rightarrow z}u_{f,g}(x)=f(z) \textnormal{ for all regular } z\in\partial\Omega
\end{align*}
referred to as the \emph{Perron solution} of the Poisson Problem with $L=\Delta$. In particular: The Poisson Problem for $L=\Delta$ and any fixed $g$ as above has a classical solution for every boundary data $f\in C(\partial\Omega,X)$ if and only if every $z\in\partial\Omega$ is regular.
\end{theorem}

Note that if $\Omega$ has a $C^{2,\alpha}$-alpha boundary, then every $z\in\partial\Omega$ is re\-gu\-lar, see e.g. \cite[Problem 2.11]{GilbargTrudingerEllipticpde}.\\

We now want to apply Schauder's Continuity Method to derive Theorem \ref{PoissonSolutionL} from Theorem \ref{PoissonSolutionLaplace}. The necessary Schauder Estimates will be derived from the real-valued case rather than proven in detail all over.

\begin{theorem}\label{KelloggSchauder}
Under the assumptions of Theorem \ref{PoissonSolutionL} a classical solution $u$ of the Poisson problem, if it exists, is an element of $C^{2,\alpha}(\overline{\Omega},X)$ and there exists a constant $C$ -- dependent on $\Omega,\lambda,\alpha$ and the $C^\alpha$-norms of the coefficients of $L$ -- such that
\begin{align*}
\|u\|_{2,\alpha}\leq C(\|u\|_{C(\overline{\Omega})}+\|F\|_{2,\alpha}+\|g\|_\alpha).
\end{align*}
Further we have the maximum principle
\begin{align*}
\|u\|_{C(\overline{\Omega},X)}\leq\|f\|_{C(\partial\Omega,X)}+C\|g\|_{C(\overline{\Omega},X)}.
\end{align*}
In particular this holds if $L=\Delta$.
\end{theorem}

\begin{proof}
The case $X=\R$ is well known, see e.g. \cite[Theorems 6.6 and 6.14, Corollary 3.8]{GilbargTrudingerEllipticpde}. Now consider a classical solution $u$ of the Poisson Problem. For every $x'\in X'$ the function $v:=\langle u,x'\rangle$ is a solution the the corresponding real-valued Poisson Problem
\begin{align*}
\begin{cases}
Lv=\langle g,x'\rangle&\textnormal{in }\Omega\\
v=\langle f,x'\rangle&\textnormal{on }\partial\Omega
\end{cases}
\end{align*}
and hence enjoys the claimed properties. Taking the supremum over all normed $x'$ it follows that
\begin{align*}
\frac{\|D_{ij}u(\xi)-D_{ij}u(\eta)\|}{|\xi-\eta|^\alpha}\leq C(\|u\|_{C(\overline{\Omega},X)}+\|f\|_{2,\alpha}+\|g\|_\alpha)
\end{align*}
hence $D_{ij}u\in C^{\alpha}(\overline{\Omega},X)$. For the first claim it remains to show that the first derivatives $D_iu$ are continuous up to the boundary. The Uniform Boundedness Principle ensures that they are bounded. It follows that $D_iu\in W^{1,p}(\Omega,X)$ for every $p\geq1$ and hence they are uniformly H\"older continuous by Morrey's Embedding Theorem, see \cite[Theorem 5.2]{ArendtKreuterMappingtheorems}. We now show the estimate for the $C^{2,\alpha}$-norm of $u$. Choose $\xi_\beta\in\overline{\Omega}$ and $x_\beta'\in \overline{B_{X'}(0,1)}$ such that
\begin{align*}
\|D^\beta u\|_{C(\overline{\Omega},X)}=\langle D^\beta u(\xi_\beta),x_\beta'\rangle
\end{align*}
holds for every multi-index $\beta$ with $|\beta|=1,2$. Using these and the above estimate we compute
\begin{align*}
\|u\|_{2,\alpha}=&\|u\|_{C(\overline{\Omega},X)}+\sum_{|\beta|=1,2}{\|D^\beta u\|_{C(\overline{\Omega},X)}}+\sum_{|\beta|=2}{\sup_{\xi,\eta\in\Omega}\frac{\|D^\beta u(\xi)-D^\beta u(\eta)\|}{|\xi-\eta|^\alpha}}\\
=&\|u\|_{C(\overline{\Omega},X)}+\sum_{|\beta|=1,2}{\|D^\beta\langle u,x_\beta'\rangle\|_{C(\overline{\Omega},\R)}}\\&+\sum_{|\beta|=2}{\sup_{\xi,\eta\in\Omega}\frac{\|D^\beta u(\xi)-D^\beta u(\eta)\|}{|\xi-\eta|^\alpha}}\\
\leq&\|u\|_{C(\overline{\Omega},X)}+\sum_{|\beta|=1,2}{\|\langle u,x_\beta'\rangle\|_{2,\alpha}}\\&+\sum_{|\beta|=2}{\sup_{\xi,\eta\in\Omega}\frac{\|D^\beta u(\xi)-D^\beta u(\eta)\|}{|\xi-\eta|^\alpha}}\\
\leq&C(\|u\|_{C(\overline{\Omega},X)}+\|f\|_{2,\alpha}+\|g\|_\alpha)
\end{align*}
where $C$ is a multiple of the original $C$ only dependant on the dimension of $\Omega$. The last estimate follows similarly from the real case.
\end{proof}

\begin{proof}[Proof of Theorem \ref{PoissonSolutionL}]
Uniqueness follows from the maximum principle. In order to show existence it suffices to consider the case $f=0$ since the Poisson Problem is equivalent to
\begin{align*}
\begin{cases}
Lv=h&\textnormal{in }\Omega\\
v=0&\textnormal{on }\partial\Omega
\end{cases}
\end{align*}
with $v\in C^{2,\alpha}(\overline{\Omega},X)$ and $h\in C^\alpha(\overline{\Omega},X)$. In fact, choose $h=g-LF$, then $u=v+F$ is the solution to the original problem. Let $L_0:=\Delta$ and $L_1:=L$. Then the operator $L_t:=(1-t)L_0+tL_1$ satisfies the same requirements as $L$ where the constants and the $C^{\alpha}$-bounds for the coefficients can be chosen independently of $t$. We have that $L_t\in\mathcal{L}(\{u\in C^{2,\alpha}(\overline{\Omega},X),u_{|\partial\Omega}=0\},C^{\alpha}(\overline{\Omega},X))$. Now if $L_tu_t=f$, then the estimates in Theorem \ref{KelloggSchauder} show that the Continuity Method \cite[Theorem 5.2]{GilbargTrudingerEllipticpde} is applicable. Since $\Omega$ is regular, Theorem \ref{PoissonSolutionLaplace} shows that  $L_0$ is bijective. Hence $L_1$ is bijective as well from which we obtain the result.
\end{proof}

For the remainder of this section we will take a quick look at the regularity of the solutions. Analogously to the Schauder estimates, we will show that the regularity carries over from the real to the vector-valued case. The Banach spaces $C^{k,\alpha}(\overline{\Omega},X)$ $(k\in\N_0)$ are defined analogously to the space $C^{2,\alpha}(\overline{\Omega},X)$. We start with the interior regularity:

\begin{proposition}\label{local regularity}
Let $k\in\N_0$ and $0<\alpha<1$. Suppose $u\in C^2(\Omega,X)$ satisfies $Lu=g$ for some $g\in C^{k,\alpha}(\Omega,X)$. Suppose further that the coefficients of $L$ are in $C^{k,\alpha}(\Omega,\R)$ and that they are bounded in this space. Then $u\in C^{k+2,\alpha}(\overline{\Omega_0},X)$ for every $\Omega_0\ssubset\Omega_1\ssubset\Omega$ and we have the estimate
\begin{align}\label{elliptic estimate}
\|u\|_{C^{k+2,\alpha}(\overline{\Omega_0},X)}\leq C(\|u\|_{C(\overline{\Omega_1},X)}+\|g\|_{C^{k,\alpha}(\overline{\Omega_1},X)}),
\end{align}
where $C=C(L,\Omega_1,\Omega_0,\alpha,d)$. In particular: If we replace $C^{k,\alpha}$ by $C^\infty$, then also $u\in C^\infty(\Omega,X)$.
\end{proposition}

\begin{proof}
The case $X=\R$ is well known, see \cite[Theorem 6.17 and Problem 6.1]{GilbargTrudingerEllipticpde}. Let $\rho_h$ be a mollifier. If $|h|<\dist(\Omega_1,\partial\Omega)$, then the functions $u_h:=\rho_h\ast u$ and $g_h:=\rho_h\ast g$ satisfy $Lu_h=g_h$ in $\Omega_1$. We have $u_h\in C^\infty(\overline{\Omega_1},X)$ and hence we may compute the norm of $u_h$ in $C^{k+2,\alpha}(\overline{\Omega_0},X)$. Proceeding as in the proof of Theorem \ref{KelloggSchauder} we obtain the estimate \eqref{elliptic estimate} for $u_h$ and $g_h$. Since $u_h\rightarrow u$ uniformly in $\overline{\Omega_1}$ and $g_h\rightarrow g$ in $C^{k,\alpha}(\overline{\Omega_1},X)$ it follows that $u_h$ is Cauchy in $C^{k+2,\alpha}(\overline{\Omega_0},X)$ and hence converges in this space. We have that $u_h\rightarrow u$ in $C^{2,\alpha}(\overline{\Omega_0},X)$ and hence $u\in C^{k+2,\alpha}(\overline{\Omega_0},X)$ and the estimate \eqref{elliptic estimate} follows as well.
\end{proof}

\begin{corollary}
In the setting of Theorem \ref{PoissonSolutionL} assume that the boundary of $\Omega$ is of class $C^{k+2,\alpha}$. Assume further that $g\in C^{k,\alpha}(\overline{\Omega},X)$ and the coefficients of $L$ are bounded in $C^{k,\alpha}(\overline{\Omega},\R)$. Assume finally that $F\in C^{k+2,\alpha}(\overline{\Omega},X)$. Then the solution $u$ to the Poisson problem is an element of $C^{k+2,\alpha}(\overline{\Omega},X)$ as well. In particular: If we replace $C^{k,\alpha}$ and $C^{k+2,\alpha}$ by $C^\infty$ then $u\in C^\infty(\overline{\Omega},X)$.
\end{corollary}

\begin{proof}
Again the real-valued case is well known, see \cite[Theorem 6.19]{GilbargTrudingerEllipticpde}. Now let $X$ be arbitrary. From Proposition \ref{local regularity} we have that $u\in C^{k+2}(\Omega,X)$. If we can show that the $(k+2)-nd$ derivatives of $u$ are H\"older-continuous, the result follows easily. To see this, note that by the real-valued case we have that $\langle u,x'\rangle\in C^{k+2,\alpha}(\overline{\Omega},\R)$ for every $x'\in X'$. Let $v$ be any $(k+2)-nd$ derivative of $u$. The quotient $\frac{v(\xi)-v(\eta)}{|\xi-\eta|^\alpha}$ is weakly bounded and hence also bounded for all $\xi,\eta\in\overline{\Omega}$. Thus $v\in C^\alpha(\overline{\Omega},X)$.
\end{proof}

\section{Application: Heat Equation on Irregular Domains}

As a second application of our work we want to investigate the heat equation with Dirichlet boundary data

\begin{align*}
\begin{cases}
\frac{d}{dt}u=\Delta u\\
u(t,\cdot)\in C_0(\overline{\Omega},X)\textnormal{ for all }t>0\\
u(0,\cdot)=u_0.\\
\end{cases}
\end{align*}

This problem is not well-posed if the domain $\Omega$ is irregular for the Dirichlet problem, see \cite{ArendtBenilanWienerRegularity}. Hence we want to impose boundary conditions which incorporate the nature of the underlying Dirichlet problem. We consider the problem

\begin{align}\label{cauchy problem}
\begin{cases}
\frac{d}{dt}u=\Delta u\\
u(t,\cdot)\in C_b(\Omega,X)\textnormal{ for all }t>0\\
\lim_{\xi\rightarrow z,\xi\in\Omega}u(t,\xi)=0\textnormal{ for all regular }z\in\partial\Omega\textnormal{ and all }t>0\\
u(0,\cdot)=u_0.
\end{cases}
\end{align}

We will treat this problem as an abstract Cauchy problem. For an overview of holomorphic semigroups and their generators, we refer to \cite[Chapter 2]{ArendtSemigroupsandavolutionequations} and \cite[Chapter 3.7]{ArendtBattyHieberNeubranderVector-valuedlaplacetransformation}. Consider the Banach space
\begin{align*}
C_{b,0,\textnormal{reg}}(\Omega,X):=\{u\in C_b(\Omega,X),\lim_{\xi\rightarrow z}u(\xi)=0\textnormal{ for all }z\in\partial_\textnormal{reg}\Omega\}.
\end{align*}
On $C_{b,0,\textnormal{reg}}(\Omega,X)$ we consider the \emph{Perron-Dirichlet-Laplacian} $\Delta_{\textnormal{PD}}$ given by
\begin{align*}
D(\Delta_{\textnormal{PD}})&:=\{u\in C_{b,0,\textnormal{reg}}(\Omega,X), \Delta u\in C_{b,0,\textnormal{reg}}(\Omega,X)\}\\
\Delta_{\textnormal{PD}}u&:=\Delta u,
\end{align*}
i.e. $\Delta_{\textnormal{PD}}$ is the distributional Laplacian with maximal domain in $C_{b,0,\textnormal{reg}}(\Omega,X)$. The main theorem of this section is

\begin{theorem}\label{PD semigroup}
The operator $\Delta_{\textnormal{PD}}$ is the generator of a bounded holomorphic semigroup on $C_{b,0,\textnormal{reg}}(\Omega,X)$.
\end{theorem}

\begin{corollary}
The problem (\ref{cauchy problem}) has a unique mild solution for all $u_0\in\overline{D(\Delta_\textnormal{PD})}$.
\end{corollary}

We will need some preparations for the proof.

\begin{lemma}
$\Delta_{\textnormal{PD}}$ is closed and $0\in\rho(\Delta_{\textnormal{PD}})$.
\end{lemma}

\begin{proof}
One easily sees that $\Delta_{\textnormal{PD}}$ is a closed operator. From the uniqueness of the Perron-solution it follows that $\Delta_{\textnormal{PD}}$ is injective. It remains to show that $\Delta_{\textnormal{PD}}$ is surjective. Let $v\in C_{b,0,\textnormal{reg}}(\Omega,X)$ and let $\tilde{v}$ be its extension to $\R^d$ by $0$. As in the scalar-valued case, see e.g. \cite[II,\S3]{DautrayLionsMathematicalAnalysis} one sees that the Newtonian potential $w$ of $\tilde{v}$ is continuous and satisfies $\Delta w=\tilde{v}$ in the sense of distributions. Let $\varphi:=w_{|\partial\Omega}$, then $u:=H_{-\varphi}+w\in C_{b,0,\textnormal{reg}}(\Omega,X)$ satisfies $\Delta u=v$.
\end{proof}

For the $m$-dissipativity of the Perron-Dirichlet-Laplacian, we will need a stronger maximum principle.

\begin{proposition}\cite[Theorem 5.2.6 (i)]{ArmitageGardinerPotentialTheory}\label{regularmaximumprinciple}
Let $u\in C_b(\Omega\cup \partial_\textnormal{reg}\Omega,\R)$ be subharmonic. Then $\sup_{\xi\in\Omega}u(\xi)=\sup_{z\in\partial_\textnormal{reg}\Omega}u(z)$. In particular if $u$ is even harmonic, then $\inf_{z\in\partial_\textnormal{reg}\Omega}u(z)\leq u(\xi)\leq\sup_{z\in\partial_\textnormal{reg}\Omega}u(z)$.
\end{proposition}

\begin{corollary}\label{regmaxprinccor}
Let $u\in C_b(\Omega\cup \partial_\textnormal{reg}\Omega,\R)$. Suppose that there exist $\lambda,M\geq0$ such that $\Delta u-\lambda u\geq0$ (in the sense of distributions) and $u\leq M$ on $\partial_\textnormal{reg}\Omega$. Then $u\leq M$.
\end{corollary}

\begin{proof}
Let $\omega:=\{\xi\in\Omega,u(\xi)\geq0\}$ and define
\begin{align*}
v(\xi):=\begin{cases}
u(\xi),&\textnormal{ if }\xi\in\omega\\
0,&\textnormal{ otherwise.}
\end{cases}
\end{align*}
In $\omega$ we have
\begin{align*}
\Delta u\geq\Delta u-\lambda u\geq 0,
\end{align*}
hence $u$ is subharmonic in $\omega$. Moreover, $v$ is constant in $(\omega^c)^\circ$ and hence subharmonic. On the boundary $\partial\omega$ we have that $v$ is constantly zero and hence $v(z)\leq\avgint_{B(z,r)}{v}$ for all $z\in\partial\omega$ and $r>0$ sufficiently small, showing that $v$ is subharmonic in $\Omega$. Note that $v$ can be continuously extended on the regular boundary points of $\Omega$ and that it satisfies $v\leq M$ on $\partial_{\textnormal{reg}}\Omega$. Hence by Proposition \ref{regularmaximumprinciple}, we have that
\begin{align*}
u\leq v\leq M.
\end{align*}
\end{proof}

\begin{proposition}
$\Delta_{\textnormal{PD}}$ is $m$-dissipative.
\end{proposition}

\begin{proof}
Let $t>0, u\in D(\Delta_{\textnormal{PD}})$ and define $M:=\|u-t\Delta_{\textnormal{PD}}u\|_\infty$. For $x'\in B_{X'}(0,1)$ let $v:=\Real\langle u,x'\rangle$. Then
\begin{align*}
(v-M)-t\Delta(v-M)\leq0
\end{align*}
and $v-M\leq 0$ on $\partial_\textnormal{reg}\Omega$. Corollary \ref{regmaxprinccor} implies that $v\leq M$. It follows that
\begin{align*}
\|u\|_\infty\leq\|u-t\Delta_{\textnormal{PD}}u\|_\infty,
\end{align*}
and hence $\Delta_{\textnormal{PD}}$ is dissipative. We know that $0\in\rho(\Delta_{\textnormal{PD}})$, thus $\Delta_{\textnormal{PD}}$ is $m$-dissipative. 
\end{proof}

\begin{proposition}
The function $z\mapsto G(z)\in\mathcal{L}(L^\infty(\R^d,X))$ where
\begin{align*}
(G(z)f)(\xi):=(4\pi z)^{-d/2}\int_{\R^d}{f(\psi)\exp(-(\xi-\psi)^2/4z)\,d\psi}
\end{align*}
for all $f\in L^\infty(\R^d,X),\xi\in\R^d$, is a bounded holomorphic semigroup on $\Sigma_\theta$ for every $\theta<\pi/2$. Its generator is the distributional Laplacian with maximal domain in $L^\infty(\R^d,X)$, that is
\begin{align*}
D(\Delta_\infty)&:=\{f\in L^\infty(\R^d,X),\Delta f\in L^\infty(\R^d,X)\}\\
\Delta_\infty f&:=\Delta f.
\end{align*}
\end{proposition}

\begin{proof}
The case $X=\C$ is well known, see e.g. \cite[2.4]{ArendtSemigroupsandavolutionequations}. For the general case note that $\|G(\cdot)\|$ is obviously bounded by 1 and that testing with $x'\in X'$ shows that $G$ is a semigroup. The functionals on $\mathcal{L}(L^\infty(\R^d,x))$ given by $T\mapsto \langle x'\circ Tf,\mu\rangle$ where $f\in L^\infty(\R^d,X), x'\in X'$ and $\mu\in L^\infty(\R^d,\R)'$ are separating. Since for scalar-valued functions the semigroup $G$ is holomorphic, it follows that $z\mapsto \langle x'\circ G(z)f,\mu\rangle=\langle G(z)(x'\circ f),\mu\rangle$ is holomorphic. By \cite[Theorem 2.1]{ArendtVector-valuedharmonic} the semigroup $G$ is holomorphic on $X$-valued functions as well. To identify the generator we fix $\lambda>0=\omega(G)$, $f\in L^\infty(\R^d,X)$  and define
\begin{align*}
R(\lambda)f:=\int_0^\infty{e^{-\lambda t}G(t)f\,dt}.
\end{align*}
Note that $\langle \Delta h,x'\rangle=\Delta\langle h,x'\rangle$ holds for every $h\in L^1_\textnormal{loc}(\R^d,X)$ and every $x'\in X'$ and that closed operators commute with integration. Thus testing with $x'$ shows that $(\lambda-\Delta_\infty)R(\lambda)f=f$ and that $\lambda-\Delta_\infty$ is injective. Since $\Delta_\infty$ is obviously closed we obtain $\lambda\in\rho(\Delta_\infty)$ and $R(\lambda,\Delta_\infty)=R(\lambda)$. This shows that $\Delta_\infty$ is the generator of $G(\cdot)$.
\end{proof}

\begin{proof}[Proof of Theorem \ref{PD semigroup}]
Since $\Delta_\infty$ generates a bounded holomorphic semigroup it follows that $\|R(\lambda,\Delta_\infty)\|\leq \frac{M}{\lambda}$ for some $M\geq0$ whenever $\Real{\lambda}>0$. Let $\Real{\lambda}>0$, then by $m$-dissipativity $\lambda\in\rho(\Delta_{\textnormal{PD}})$. Let $f\in C_{b,0,\textnormal{reg}}(\Omega,X)$ and denote by $\tilde{f}\in L^\infty(\R^d,X)$ the extension of $f$ by $0$. Let $\tilde{g}:=R(\lambda,\Delta_\infty)\tilde{f}$ and $g:=R(\lambda,\Delta_{\textnormal{PD}})f$. As in the real-valued case, c.f. \cite[II,\S3]{DautrayLionsMathematicalAnalysis}, we see that $\tilde{g}\in C(\R^d,X)$ and hence $v:=g-\tilde{g}\in C_b(\Omega\cup\partial_\textnormal{reg}\Omega,X)$ satisfies $\lambda v-\Delta v=0$ and $v=-\tilde{g}$ on $\partial_\textnormal{reg}\Omega$. The Regular Maximum Principle shows that
\begin{align*}
\|v\|_\infty\leq\sup_{\partial_\textnormal{reg}\Omega}\|\tilde{g}\|\leq\frac{M}{\lambda}\|f\|_\infty
\end{align*}
and hence
\begin{align*}
\|g\|_\infty\leq\frac{2M}{\lambda}\|f\|_\infty,
\end{align*}
which finishes the proof.
\end{proof}

\begin{remark}
\begin{compactenum}[(a)]
\item If $\Omega$ is a regular domain, then $C_{b,0,\textnormal{reg}}(\Omega,X)$ is nothing but $C_0(\Omega,X)$. It follows from the density of the test functions that the semigroup generated by $\Delta_\textnormal{PD}$ is strongly continuous.
\item In general, the semigroup generated by $\Delta_\textnormal{PD}$ is not strongly continuous. Indeed, consider $d=2$, $\Omega=B(0,1)\backslash\{0\}$ and $X=\R$. We show that every $u\in D(\Delta_\textnormal{PD})$ can be continuously extended to $B(0,1)$ which implies that $D(\Delta_\textnormal{PD})$ is not dense in $C_{b,0,\textnormal{reg}}(\Omega,\R)$. To see this consider a function $u\in D(\Delta_\textnormal{PD})$. We may extend $u$ to $\R^2$ by $0$ outside of $\Omega$ and consider the tempered distribution $T_u$ defined via $\left\langle T_u,\varphi\right\rangle=\int_{\R^2}{u\varphi}$ and proceed analogously for $\Delta u$. It follows that the distribution $\Delta T_u-T_{\Delta u}$ is a distribution of order at most $2$ which is supported in $\{0\}$, i.e. $\Delta T_u-T_{\Delta u}=\sum_{|\alpha|\leq2}{a_\alpha\partial^\alpha\delta_0}$. Let $v$ and $w$ be the solutions to $\Delta v=T_{\Delta u}$ and $\Delta w=\sum_{|\alpha|\leq2}{a_\alpha\partial^\alpha\delta_0}$, then $T_u=v+w$ up to a perturbation by a harmonic function. By elliptic regularity we find that $v$ is continuous even in $0$. Further $w$ -- up to a perturbation by an analytic function -- is given by
\begin{align*}
w(x,y)=\ &a_0\log r+a_1\frac{x}{r^2}+a_2\frac{y}{r^2}\\
&+a_{11}\frac{r^2-2x^2}{r^4}+a_{22}\frac{r^2-2y^2}{r^4}+a_{12}\frac{-2xy}{r^4},
\end{align*}
where $r=\sqrt{x^2+y^2}$. It is easy to see, that either $w$ is unbounded or $w=0$. Since $u$ is bounded, we obtain that the latter is true and hence $u$ can be continuously extended in $\{0\}$.
\end{compactenum}
\end{remark}

\begin{appendix}
\section{The vector lattice $\mathcal{H}_b$}\label{harmonic lattice appendix}
An ordered vector space which has the interpolation property is automatically an order complete vector lattice, i.e. every set bounded from above has a supremum. To see this let $A$ be a set bounded from above and $B$ be the set of its upper bounds. Since $a\leq b$ for all $a\in A, b\in B$ the interpolation property yields an element $w$ satisfying $a\leq w\leq b$ for all $a\in A, b\in B$, i.e. $w$ is the least upper bound of $A$. In particular in the setting of Lemma \ref{interpolation property lemma} the vector space $\mathcal{H}_b(\Omega,X)$ is an order complete vector lattice with respect to the order on $C(\Omega,X)$. Note however that it is not a sublattice of $C(\Omega,X)$ since the pointwise supremum of two harmonic functions is not harmonic in general. For two vectors $u,v\in\mathcal{H}_b(\Omega,\R)$ we denote their supremum in the vector lattice $\mathcal{H}_b(\Omega,\R)$ by $u\vee_\mathcal{H}v$. In the case $X=\R$, the fact that $\mathcal{H}_b(\Omega,\R)$ is an order complete vector lattice can also be found in a more general form in \cite[Part IV, Theorem 11]{BrelotLecturesPotentialTheory}. Moreover:

\begin{theorem}\label{harmonicfunctionsareabanachlattice}
$\mathcal{H}_b(\Omega,\R)$ is an order complete Banach lattice.
\end{theorem}

For the proof we will need

\begin{lemma}\label{harmonicsupremumconstruction}
Let $u,v\in\mathcal{H}_b(\Omega,\R)$. Choose $\omega_n\ssubset\omega_{n+1}\ssubset\Omega$ such that $\bigcup_n\omega_n=\Omega$ and such that $\omega_n$ is regular for each $n\in\N$. For all $\xi\in\partial\omega_n$ set
\begin{align*}
f_n(\xi):=u(\xi)\vee_\R v(\xi),
\end{align*}
where $u(\xi)\vee_\R v(\xi)$ denotes the supremum in $\R$. Let $h_n$ be the solution to the Dirichlet problem in $\omega_n$ with boundary data $f_n$. Then $h_n\rightarrow u\vee_\mathcal{H}v$ uniformly on compact sets.
\end{lemma}

\begin{proof}
The fact that $\mathcal{H}_b(\Omega,\R)$ is an order complete vector lattice was discussed in the introduction of the appendix. It remains to show that $h_n$ converges to $u\vee_\mathcal{H}v$. First let $K\subset\Omega$ be compact. We may without loss of generality assume that $K\subset\omega_1$. For $n\in\N$ and $\xi\in\partial\omega_n$ we have
\begin{align*}
h_n(\xi)=u(\xi)\vee_\R v(\xi)\leq\left(u\vee_\mathcal{H}v\right)(\xi).
\end{align*}
Note that the functions on the left hand side and on the right hand side are harmonic while the one in the middle is subharmonic. The maximum principle implies that
\begin{align*}
u(\xi)\vee_\R v(\xi)\leq h_n(\xi)\leq\left(u\vee_\mathcal{H}v\right)(\xi)
\end{align*}
for all $\xi\in\omega_n$. In particular this holds for $\xi\in K$. The Poisson integral formula implies that $h_n$ is equicontinuous, hence the Arzela-Ascoli Theorem yields a subsequence of $h_n$ which converges uniformly on $K$ to a function $h\in\mathcal{H}_b(K^\circ,\R)$.\\
Next take $K_n:=\overline{\omega_n}$ then using a diagonal argument we find a subsequence of $h_n$ which converges to $h\in\mathcal{H}_b(\Omega,\R)$ uniformly on every $K_n$. The above inequality shows that $h=u\vee_\mathcal{H}v$. A subsequence argument shows that $h_n\rightarrow u\vee_\mathcal{H}v$.
\end{proof}

One can easily use Lemma \ref{harmonicsupremumconstruction} to show that this construction works analogously for other lattice operations:

\begin{corollary}\label{harmonicsupremumconstructioncorollary}
In the formulation of Lemma \ref{harmonicsupremumconstruction} choose
\begin{align*}
f_n(\xi):=|u(\xi)|_\R,
\end{align*}
where $|u(\xi)|_\R$ denotes the absolute value in $\R$. Then $h_n\rightarrow |u|_\mathcal{H}$ uniformly on compact sets, where $|u|_\mathcal{H}$ denotes the absolute value in the lattice $\mathcal{H}_b(\Omega,\R)$.
\end{corollary}

\begin{proof}[Proof of Theorem \ref{harmonicfunctionsareabanachlattice}]
It remains to show that $\|\,|u|_\mathcal{H}\,\|_\infty=\|u\|_\infty$. The estimate \glqq$\geq$\grqq\ is immediate. On the other hand if $h_n$ is chosen as in Corollary \ref{harmonicsupremumconstructioncorollary}, then the maximum principle implies that
\begin{align*}
0\leq h_n(\xi)\leq\max_{\xi\in\partial\omega_n}|u(\xi)|_\R\leq\|u\|_\infty.
\end{align*}
Since $h_n(\xi)$ converges to $|u(\xi)|_\mathcal{H}$ the estimate \glqq$\leq$\grqq\ follows.
\end{proof}
\end{appendix}

\def\cprime{$'$}
\providecommand{\bysame}{\leavevmode\hbox to3em{\hrulefill}\thinspace}
\providecommand{\MR}{\relax\ifhmode\unskip\space\fi MR }
\providecommand{\MRhref}[2]{%
  \href{http://www.ams.org/mathscinet-getitem?mr=#1}{#2}
}
\providecommand{\href}[2]{#2}

\end{document}